\documentclass{amsart}
\usepackage{amsmath, amsfonts}
\usepackage{amsthm}
\usepackage{amssymb}

\newtheorem{theorem}{Theorem}
\newtheorem{lemma}{Lemma}
\newtheorem{claim}{Claim}[section]

\newtheorem{corollary}{Corollary}

\theoremstyle{definition}

\theoremstyle{remark}

\numberwithin{equation}{section}

\begin{document}

\title{Ramsey numbers of trees versus odd cycles}

\author{Matthew Brennan}
\address{Massachusetts Institute of Technology, Cambridge, MA 02139}
\email{brennanm@mit.edu}


\maketitle

\begin{abstract}
Burr, Erd\H{o}s, Faudree, Rousseau and Schelp initiated the study of Ramsey numbers of trees versus odd cycles, proving that $R(T_n, C_m) = 2n - 1$ for all odd $m \ge 3$ and $n \ge 756m^{10}$, where $T_n$ is a tree with $n$ vertices and $C_m$ is an odd cycle of length $m$. They proposed to study the minimum positive integer $n_0(m)$ such that this result holds for all $n \ge n_0(m)$, as a function of $m$. In this paper, we show that $n_0(m)$ is at most linear. In particular, we prove that $R(T_n, C_m) = 2n - 1$ for all odd $m \ge 3$ and $n \ge 25m$. Combining this with a result of Faudree, Lawrence, Parsons and Schelp yields $n_0(m)$ is bounded between two linear functions, thus identifying $n_0(m)$ up to a constant factor.
\end{abstract}

\section{Introduction}

The generalized Ramsey number $R(H, K)$ is the smallest positive integer $N$ such for any graph $G$ with at least $N$ vertices either $G$ contains $H$ as a subgraph or its complement $\overline{G}$ contains $K$ as a subgraph, where $H$ and $K$ are any two given graphs. When $H$ and $K$ are complete graphs with $m$ and $n$ vertices respectively, $R(H, K)$ is the classical Ramsey number $R(m,n)$. Classical Ramsey numbers are notoriously difficult to determine. The exact values of many small classical Ramsey numbers including $R(5, 5)$ remain unknown. Because of this, Chv\'{a}tal and Harary proposed to study generalized Ramsey numbers of graphs that are not complete in a series of papers in the early 1970's \cite{chvatal1972generalized1, chvatal1972generalized2, chvatal1973generalized}. 

Generalized Ramsey numbers have since been well studied for a variety of graphs, including trees and odd cycles. Let $T_n$ be a tree with $n$ vertices and $C_m$ denote a cycle of length $m$. Bondy and Erd\H{o}s showed that $R(C_n, C_n) = 2n - 1$ for odd $n$ and $R(C_n, C_{2r-1}) = 2n - 1$ if $n>r(2r-1)$ \cite{bondy1973ramsey}. Chv\'{a}tal identified the Ramsey numbers of trees versus complete graphs, showing that $R(T_n, K_m) = (n - 1)(m - 1) + 1$ for all positive integers $m$ and $n$ \cite{chvatal1977tree}. Faudree, Lawrence, Parsons and Schelp identified the Ramsey numbers of paths versus odd cycles. If $P_n$ denotes a path on $n$ vertices, they showed that $R(P_n, C_m) = 2n - 1$ for $n \ge m \ge 3$ and $R(P_n, C_m) = \max \{ 2n - 1, m + \lfloor n/2 \rfloor - 1 \}$ for $m \ge n \ge 2$ where $m$ is odd \cite{faudree1974path}. Faudree, Schelp and Simonovits showed several bounds and exact results on the Ramsey numbers $R(T_n, C_{\ge m})$ where $C_{\ge m}$ denotes the family of cycles of length at least $m$ in \cite{faudree1990some}. These results include that $R(T_n, C_{\ge m}) \le 2m + 2n - 7$ for all $m,n \ge 3$, $R(T_n, C_{\ge m}) \le m + n - 2$ if either $m \ge n$ or $n \ge 432 m^6 - m^2$, and $R(T_n, C_{\ge m}) = n + \lfloor m/2 \rfloor - 1$ if $T_n$ is a tree with maximum degree less than $n - 3m^2$ and $n \ge 432 m^6$. A survey of results about generalized Ramsey numbers can be found in \cite{radziszowski1994small}.

There have also been lower bounds shown to hold for generalized Ramsey numbers of all graphs. In 1981, Burr showed a lower bound for $R(H,K)$ in terms of the chromatic number $\chi(K)$ of a graph $K$ and its chromatic surplus $s(K)$ -- the minimum number of vertices in a color class over all proper vertex colorings of $K$ using $\chi(K)$ colors.

\begin{theorem}[Burr \cite{burr1981ramsey}]
If $s(K)$ is the chromatic surplus of the graph $K$, then for all connected graphs $H$ with $n \ge s(K)$ vertices we have
$$R(H,K) \ge (n-1)(\chi(K) - 1) + s(K).$$
\end{theorem}

In the case of several of the Ramsey numbers mentioned above, Burr's lower bound is tight. For $K = C_m$, Burr's lower bound yields that $R(H, C_m) \ge 2n - 1$ where $n = |V(H)|$ since $\chi(C_m) = 3$ and $s(C_m) = 1$. In 1982, Burr, Erd\H{o}s, Faudree, Rousseau and Schelp showed that for sufficiently large $n$ and small $\epsilon$, Burr's lower bound on the Ramsey numbers of sparse connected graphs $G$ with at most $(1 + \epsilon)n$ edges versus odd cycles $C_m$ is tight. Specifically, they proved the following theorem.

\begin{theorem}[Burr et al. \cite{burr1982ramsey}]
If $G$ is a connected graph with $n$ vertices and at most $n(1 + 1/42m^5)$ edges where $m \ge 3$ is odd and $n \ge 756m^{10}$, then $R(G, C_m) = 2n - 1$.
\end{theorem}

This theorem implies the following corollary identifying the Ramsey number of trees versus odd cycles for $n$ very large relative to $m$.

\begin{corollary}[Burr et al. \cite{burr1982ramsey}]
$R(T_n, C_m) = 2n - 1$ for all odd integers $m \ge 3$ and integers $n \ge 756m^{10}$.
\end{corollary}

Burr, Erd\H{o}s, Faudree, Rousseau and Schelp asked what the minimum positive integer $n_0(m)$ such that this result holds for all $n \ge n_0(m)$ is as a function of $m$. Their corollary shows that $n_0(m)$ is at most a tenth-degree polynomial in $m$. We provide a new approach to examining the Ramsey numbers of trees versus odd cycles and improve this bound, showing that $n_0(m)$ is at most linear in $m$. In particular, we prove the following theorem.

\begin{theorem}
$R(T_n, C_m) = 2n - 1$ for all odd integers $m \ge 3$ and integers $n \ge 25m$.
\end{theorem}

The result of Faudree, Lawrence, Parsons and Schelp that $R(P_n, C_m) = \max \{ 2n - 1, m + \lfloor n/2 \rfloor - 1 \}$ for $m \ge n \ge 2$ where $m$ is odd shows $n_0(m) \ge 2m/3 - 1$ \cite{faudree1974path}. Combining this with Theorem 3 yields that $n_0(m)$ is bounded between two linear functions.

In the next two sections, we prove Theorem 4. We first provide the key lemmas that we use in our proof and then present the proof through a sequence of claims. An important remark is that Burr's lower bound in the case of trees versus odd cycles can be shown by considering the complete bipartite graph $K_{n - 1, n - 1}$. Because it is bipartite, it does not contain the odd cycle $C_m$ as a subgraph. Furthermore, $\overline{K_{n - 1, n - 1}}$ consists of two connected components of size $n - 1$ and therefore does not contain $T_n$ as a subgraph. This extremal graph $K_{n - 1, n - 1}$ will be useful in motivating our proof of Theorem 4, the last steps of which are devoted to showing that any graph that is any counterexample to Theorem 4 would necessarily have a similar structure to $K_{n - 1, n - 1}$.

\section{Preliminaries and Lemmas}

We first provide the notation we will adopt on proving Theorem 4. Given a graph $G$, $d_X(v)$ denotes the degree of a vertex $v$ in a set $X \subseteq V(G)$ in $G$ and $\overline{d_X}(v)$ denotes the degree of $v$ in $X$ in $\overline{G}$, the complement graph of $G$. We similarly let $N_X(v)$ and $\overline{N_{X}}(v)$ denote the sets of neighbors of $v$ in the set $X$ in $G$ and $\overline{G}$, respectively. Note that $d_X(v) + \overline{d_X}(v) = |X|$, $d_X(v) = |N_X(v)|$ and $\overline{d_{X}}(v) = |\overline{N_X}(v)|$ for any set $X \subseteq V(G)$ with $v \not \in X$. When the set $X$ is omitted, $X$ is implicitly $V(G)$ where $G$ is the graph in which the vertex $v$ lies. We denote the maximum and minimum degrees of a graph $H$ by $\Delta(H)$ and $\delta(H)$, respectively. Given a subset $S \subseteq V(G)$, we denote the subgraph of $G$ induced by the set $S$ by $G[S]$.

We now present several lemmas that will be used throughout the proof of Theorem 4. The first two lemmas are extensions of classical results.

\begin{lemma}
Let $F$ be a forest with $k$ connected components. Let $w_1, w_2, \dots, w_k \in V(F)$ be vertices from distinct connected components of $F$. Let $H$ be a graph with $\delta(H) \ge |V(F)| - 1$ and $u_1, u_2, \dots, u_k$ be distinct vertices in $V(H)$. Then $F$ can be embedded in $H$ such that $w_i$ is mapped to $u_i$ for all $1 \le i \le k$.
\end{lemma}

\begin{proof}
Begin by mapping $w_i$ to $u_i$ for all $1 \le i \le k$. Greedily extend this embedding as follows: if $x \in V(F)$ has not been embedded to $H$ but a neighbor $y \in N_F(x)$ has been embedded to $z \in V(H)$ then map $x$ to a vertex in $N_H(z)$ that has not been embedded to, if such a vertex exists. Since initially a vertex from each connected component of $F$ is mapped to $H$, if all of $F$ has not yet been embedded to $H$ then there must be such a vertex $x$ adjacent to a vertex $y$ that has already been embedded. For the described embedding to fail, there must be a point in this procedure when at most $|V(F)| - 1$ vertices have been embedded to and the embedding cannot be extended. Therefore all of $N_H(z) \cup \{ z \}$ has been embedded to for some $z \in V(H)$. However $|N_H(z) \cup \{ z \}| \ge \delta(H) + 1 \ge |V(F)|$, which is a contradiction. Therefore the embedding succeeds, proving the lemma.
\end{proof}

The next lemma is also included in our simultaneous work on Ramsey numbers of trees and unicyclic graphs versus fans in \cite{brennan}, where it appears as Lemma 2.

\begin{lemma}
Given a tree $T$ with $|V(T)| \ge 3$, there exists a vertex $v \in V(T)$ satisfying that the vertices of the forest $T - v$ can be partitioned into two disjoint sets $K$ and $H$ such that there are no edges between $K$ and $H$ and
$$\frac{1}{3}(|V(T)|-1) \le |K|, \, |H| \le \frac{2}{3} (|V(T)|-1).$$
\end{lemma}

\begin{proof}
Note that for any vertex $v \in V(T)$, the forest $T - v$ has $d_T(v)$ connected components. Now consider the following procedure. Set $v$ initially to be an arbitrary leaf of $T$. At each step, if $T - v$ has a connected component $C$ with $|C| \ge |V(T)|/2$, set $v$ to the unique neighbor $u$ of $v$ in $C$. If no such connected component $C$ exists, terminate the procedure. Observe that $T - u$ has a connected component of size $|V(T)| - |C|$ and one or more connected components with the sum of their sizes equal to $|C| - 1$. Therefore either $|C| = |V(T)|/2$ or the size of the largest connected component of $T - v$ decreases on setting $v$ to $u$. This implies that either the procedure leads to a vertex $v$ such that $T - v$ has a connected component $C$ of size $|C| = |V(T)|/2$ or terminates at a vertex $v$ such that all connected components of $T - v$ are strictly smaller than $|V(T)|/2$.

If $v$ satisfies that $T - v$ has a connected component $C$ of size $|C| = |V(T)|/2$, then let $K = C$ and $H = V(T - v) - C$. We have that $|H| = |V(T)|/2 - 1$ and $|K| = |C| = |V(T)|/2$, which implies the lemma since $|V(T)|$ must be even and hence $|V(T)| \ge 4$. Note $d_T(v) = 1$ is impossible because of the condition on $v$ and since $|V(T)| \ge 3$. If $d_T(v) = 2$, then $T - v$ has two connected components $K$ and $H$ with $|K| + |H| = |V(T)| - 1$. Since $|K|, |H| < |V(T)|/2$, it holds that $|K| = |H| = (|V(T)|-1)/2$, in which case the lemma is also true.

Now consider the case in which $d_T(v) = d \ge 3$. Let the connected components of $T - v$ be $C_1, C_2, \dots, C_d$. Here, it must hold that $|V(T)| \ge 4$. Without loss of generality assume that $|C_1| \le |C_2| \le \cdots \le |C_d| < |V(T)|/2$. Since $d \ge 3$ and $|C_1| + |C_2| + \cdots + |C_d| = |V(T)| - 1$, we have that $|C_1| \le (|V(T)| - 1)/3$. Let $t$ be the largest integer such that $|C_1| + |C_2| + \cdots + |C_t| \le (|V(T)|-1)/3$ holds. If $t = d - 1$, then $|C_d| \ge 2(|V(T)|-1)/3$ which is impossible because $|C_d| < |V(T)|/2$. This implies that $t \le d - 2$. By definition, $|C_1| + |C_2| + \cdots + |C_{t+1}| > (|V(T)|-1)/3$. If $|C_1| + |C_2| + \cdots + |C_{t+1}| \le 2(|V(T)|-1)/3$, then letting $K = C_1 \cup C_2 \cup \cdots \cup C_{t+1}$ and $H = C_{t+2} \cup C_{t+3} \cup \cdots \cup C_d$ yields valid sets $K$ and $H$. If $|C_1| + |C_2| + \cdots + |C_{t+1}| > 2(|V(T)|-1)/3$, then $(|V(T)|-1)/3 < |C_{t+1}| < |V(T)|/2 \le 2(|V(T)|-1)/3$. In this case, letting $K = C_{t+1}$ and $H = C_1 \cup \cdots \cup C_t \cup C_{t+2} \cup \cdots \cup C_d$ yields the desired sets $K$ and $H$. This proves the lemma.
\end{proof}

Given a graph $G$ and positive integer $n$, let $ex(n, G)$ be the maximum number of edges that a graph on $n$ vertices can have without containing $G$ as a subgraph. This last lemma by Erdos and Galliai determines $ex(n, P_k)$, where $P_k$ is a path with $k$ edges and $k + 1$ vertices. It can be found in \cite{erdosgallai} as Theorem 2.6.

\begin{lemma}[Erdos and Gallai \cite{erdosgallai}]
For all positive integers $n$ and $k$,
$$ex(n, P_k) \le \frac{n(k-1)}{2}$$
where equality holds if and only if $k$ divides $n$, in which case the only graph achieving equality is a union of $\frac{n}{k}$ disjoint copies of $K_k$.
\end{lemma}

\section{Proof of Theorem 4}

Let $G$ be a graph with $2n - 1$ vertices and assume for contradiction that $G$ does not contain $C_m$ as a subgraph and $\overline{G}$ does not contain $T_n$ as a subgraph, where $n \ge 25m$. The result $R(T_n, K_m) = (n-1)(m-1) + 1$ in \cite{chvatal1977tree} applied with $m = 3$ yields that $R(T_n, C_3) = 2n - 1$ for all $n$. Therefore, it suffices to prove the result in the case when $m \ge 5$.

Our proof of Theorem 4 uses the following key ideas. The lack of a tree in $\overline{G}$ guarantees a large degree vertex in $G$. The absence of an $m$-cycle in $G$ along with this high degree vertex implies there is no path with $m - 1$ vertices among its neighbors, which is highly restrictive. The resulting constraints along with two methods for embedding trees yield that there are two large sets $S_1$ and $S_2$ of vertices in $G$ with a large fraction of the edges between them present. The remainder of the proof is devoted to showing that $G$ must have a similar structure to the extremal graph $K_{n-1, n-1}$. The lack of a length $m$ cycle alternating between these two sets shows that $S_1 \cup S_2$ induces a bipartite subgraph of $G$ and imposes significant constraints on vertices not in $S_1$ and $S_2$, which are enough to yield a contradiction.

We now proceed to present the proof of Theorem 4 through a series of claims. The first claim bounds the number of edges in a set of neighbors of a vertex. The second claim uses this bound to guarantee that a set of neighbors of a vertex contains a large subset that induces a subgraph with high minimum degree.

\begin{claim}
For any vertex $v \in V(G)$ and set $X \subseteq V(G)$, the number of edges in the induced subgraph $G[N_X(v)]$ is at most $\frac{1}{2}(m-3)d_X(v)$.
\end{claim}

\begin{proof}
Assume for contradiction that the number of edges in $G[N_X(v)]$ is greater than $\frac{1}{2} (m - 3)d_X(v)$. By Lemma 7, there is a path $P$ in $G[N_X(v)]$ with $m - 2$ edges and $m - 1$ vertices. Since $v$ is adjacent to all vertices of $P$, the vertices $\{v\} \cup P$ form a cycle of length $m$, which is a contradiction. The claim follows.
\end{proof}

\begin{claim}
For any vertex $v \in V(G)$, subset $X \subseteq V(G)$ and positive real number $D$, there is a subset $S \subseteq N_X(v)$ such that
$$|S| \ge \left( 1 - \frac{(m - 3)}{2D} \right) d_X(v) \quad \text{and} \quad \delta(\overline{G[S]}) \ge |S| - D.$$
\end{claim}

\begin{proof}
Consider the following procedure. Begin by setting $R = N_X(v)$. At each step, if there is a vertex $u \in R$ such that $d_{R}(u) \ge D$, then set $R = R \backslash \{ u \}$, otherwise terminate the procedure. At each step, it follows that the number of edges in $G[R]$ decreases by at least $D$. Let $S$ denote the set $R$ obtained once the procedure has terminated. By Claim 8, the number of edges in $G[R]$ begins no greater than $\frac{1}{2}(m - 3)d_X(v)$. Therefore the number of steps $t$ of the procedure satisfies that $t \le (m - 3)d_X(v)/2D$. Since $t$ vertices are removed in this procedure, we have that $S$ has size
$$|S| = d_X(v) - t \ge \left( 1 - \frac{(m - 3)}{2D} \right) d_X(v).$$
Furthermore, for this procedure to terminate it must follow that $\Delta(G[S]) \le D - 1$. This implies that
$$\delta(\overline{G[S]}) = |S| - 1 - \Delta(G[S]) \ge |S| - D.$$
Therefore $S$ has the desired properties, completing the proof of the claim.
\end{proof}

If $\delta(\overline{G}) \ge n - 1$ then by Lemma 5, $T_n$ can be embedded into $\overline{G}$. Therefore $\delta(\overline{G}) \le n - 2$ which implies that $\Delta(G) = 2n - 2 - \delta(\overline{G}) \ge n$. Let $v \in V(G)$ a maximum degree vertex of $G$ with $d(v) = \Delta(G) \ge n$. Applying Claim 9 with $D = \sqrt{\frac{1}{2}(m-3)n}$ yields that there is a subset $S_1 \subseteq N(v)$ such that
$$|S_1| \ge \left( 1 - \frac{(m - 3)}{2D} \right) d(v) \ge n - \sqrt{\frac{1}{2}(m-3)n}.$$
Furthermore it follows that
$$\delta(\overline{G[S_1]}) \ge |S_1| - D \ge n - \sqrt{2(m-3)n}.$$
Note that the choice $D = \sqrt{\frac{1}{2}(m-3)n}$ maximizes this lower bound on $\delta(\overline{G[S_1]})$. Let $O_1 = V(G) \backslash S_1$. The next claim is the key ingredient to show that there is another large set $S_2$ analogous to $S_1$ and disjoint from $S_1$ in $G$. The proof of this claim applies a method to greedily embed trees used in the proof of Claim 3.3 in \cite{brennan}.

\begin{claim}
There is a vertex $u \in V(G)$ such that
$$d_{O_1}(u) \ge n - \sqrt{\frac{1}{2}(m - 3)n} + 1.$$
\end{claim}

\begin{proof}
By Lemma 5, any sub-forest of $T_n$ of size at most $\delta(\overline{G[S_1]}) + 1$ can be embedded in $\overline{G[S_1]}$. Since $\overline{G}$ does not contain $T_n$ as a subgraph, it must follow that $\delta(\overline{G[S_1]}) \le n - 2$. Note that $T_n$ has a connected subtree $H$ on $\delta(\overline{G[S_1]}) + 1$ vertices. For instance, such a subtree is obtained by removing a leaf of $T_n$ and repeatedly removing a leaf of the resulting tree until exactly $\delta(\overline{G[S_1]}) + 1$ vertices remain. By Lemma 5, $H$ can be embedded to $\overline{G[S_1]}$. Let $R \subseteq S_1$ denote the set of vertices that $V(H)$ is mapped to under this embedding.

We now define a procedure to greedily extend this embedding of $H$ to an embedding of $T_n$ in $\overline{G}$. At any point in this procedure, let $K$ denote the subgraph of $\overline{G}$ induced by the set of vertices that have so far been embedded to. Initially, $V(K) = R$ and, throughout the procedure, $R \subseteq V(K)$ remains true. Now extend this embedding by repeating the following: if $w_1, w_3 \in V(T_n)$ and $w_2 \in V(G)$ satisfy that (1) $w_1 \in V(T_n)$ has been mapped to $w_2 \in V(K)$, (2) $w_3 \in V(T_n)$ has not been embedded to $\overline{G}$ and (3) $w_3$ is adjacent to $w_1$ in $T_n$, then map $w_3$ to some vertex in $\overline{N_{V(G) \backslash V(K)}}(w_2)$ if it is non-empty. Since $T_n$ contains no cycles and $K$ remains connected throughout this procedure, each $w_3 \in V(T_n)$ that has not yet been embedded to $K$ has at most one neighbor among the vertices $V(T_n)$ that have been embedded to $K$. Furthermore, since $T_n$ is connected, if not all of $T_n$ has been embedded to $\overline{G}$ then some such $w_3 \in V(T_n)$ satisfying (1)$-$(3) must exist. Thus this embedding only fails if $\overline{N_{V(G) \backslash V(K)}}(w_2)$ is empty for some $w_2 \in V(K)$ at some point in the procedure.

Since $\overline{G}$ does not contain $T_n$ as a subgraph, this embedding procedure must fail. Therefore $\overline{N_{V(G) \backslash V(K)}}(w_2) = \emptyset$ for some $w_2 \in V(K)$ where $|V(K)| \le n - 1$ at some point during the procedure. Since $R \subseteq V(K)$, it follows that $V(G) \backslash V(K) \subseteq V(G) \backslash R$ and therefore
$$d_{V(G) \backslash R}(w_2) \ge d_{V(G) \backslash V(K)}(w_2) \ge |V(G) \backslash V(K)| \ge n.$$
Now note that Claim 9 guarantees that
$$|S_1 \backslash R| \le |S_1| - \delta(\overline{G[S_1]}) - 1 \le D - 1 = \sqrt{\frac{1}{2}(m-3)n} - 1.$$
Combining this bound with the previous inequality yields that
$$d_{O_1}(w_2) \ge d_{V(G) \backslash R}(w_2) - |S_1 \backslash R| \ge n - \sqrt{\frac{1}{2}(m -3)n} + 1.$$
Therefore $w_2$ is a vertex with the desired properties, proving the claim.
\end{proof}

From this point forward in the proof, let $d = d_{O_1}(u)$ where $u$ is the vertex guaranteed by Claim 10 and $d \ge n - \sqrt{\frac{1}{2}(m-3)n} + 1$. Now applying Claim 9 with $D = \sqrt{\frac{1}{2}(m-3)d}$ yields that there is a subset $S_2 \subseteq N_{O_1}(u)$ such that
$$|S_2| \ge \left( 1 - \frac{(m - 3)}{2D} \right) d = d - \sqrt{\frac{1}{2}(m-3)d}.$$
Furthermore it follows that
$$\delta(\overline{G[S_2]}) \ge |S_2| - D \ge d - \sqrt{2(m-3)d}.$$
Note that since $S_2 \subseteq N_{O_1}(u) \subseteq O_1$, it necessarily follows that $S_1$ and $S_2$ are disjoint. We remark that $S_1$ and $S_2$ are symmetric other than the fact that the lower bounds on $|S_2|$ and $\delta(\overline{G[S_2]})$ are weaker than those on $|S_1|$ and $\delta(\overline{G[S_1]})$. In order to obtain the bound $n \ge 25m$, we do not treat $S_1$ and $S_2$ completely symmetrically, which would entail discarding the better lower bounds for $S_1$. The next claim shows that a large fraction of the edges are present between $S_1$ and $S_2$.

\begin{claim}
Each $w \in S_1$ satisfies that $\overline{d_{S_2}}(w) < n - \delta(\overline{G[S_1]}) - 1$ and each $w \in S_2$ satisfies that $\overline{d_{S_1}}(w) < n - \delta(\overline{G[S_2]}) - 1$.
\end{claim}

\begin{proof}
Before proving this claim, we first show the two inequalities
$$(n - 1)/2 \le \delta(\overline{G[S_2]}) \quad \text{and} \quad 2(n - 1)/3 \le \delta(\overline{G[S_1]}).$$
Here we apply the lower bound on $n$ in terms of $m$. In particular, if $n \ge 18m$, we have the following inequalities
\begin{align*}
\delta(\overline{G[S_1]}) &\ge n - \sqrt{2(m-3)n} > 2n/3  > 2(n - 1)/3, \quad \text{and} \\
d &\ge n - \sqrt{\frac{1}{2}(m-3)n} + 1 > 5n/6.
\end{align*}
Applying $d > 5n/6 \ge 15m$ to the lower bound on $\delta(\overline{G[S_2]})$ yields that
\begin{align*}
\delta(\overline{G[S_2]}) &\ge d - \sqrt{2(m-3)d}> \left( 1 - \frac{2}{\sqrt{30}} \right) d \\
&> \frac{5}{6} \left( 1 - \frac{2}{\sqrt{30}} \right) n > (n-1)/2
\end{align*}
since $\frac{5}{6} \left( 1 - \frac{2}{\sqrt{30}} \right) \approx 0.53$. We now proceed to the proof of the claim.

We first show that each $w \in S_2$ satisfies that $\overline{d_{S_1}}(w) < n - \delta(\overline{G[S_2]}) - 1$. Assume for contradiction that some $w \in S_2$ satisfies that $\overline{d_{S_1}}(w) \ge n - \delta(\overline{G[S_2]}) - 1$. By Lemma 6, there is some vertex $x \in V(T_n)$ such that there is a partition $K \cup H$ of the vertices of the forest $T_n - x$ such that there are no edges between $K$ and $H$ in $T_n$ and $(n-1)/3 \le |K|, |H| \le 2(n - 1)/3$. Without loss of generality assume that $|H| \le (n - 1)/2 \le |K|$.

First suppose that $d_{K}(x) \le \overline{d_{S_1}}(w)$. Consider the following embedding of $T_n$ into $\overline{G}$. Note that since $K$ and $H$ are unions of connected components of $T_n - x$, it follows that $H \cup \{ x \}$ is a subtree of $T_n$. Since
$$|H \cup \{ x \}| \le 1 + (n - 1)/2 \le \delta(\overline{G[S_2]}) + 1,$$
by Lemma 5 we have that $H \cup \{ x \}$ can be embedded in $\overline{G[S_2]}$ such that $x$ is mapped to $w$. Furthermore, note that $K$ is the union of connected components of $T_n - x$, $K$ is a sub-forest of $T_n$ and $N_K(x)$ consists of exactly one vertex from each of the connected components of $K$. Now note that since
$$|K| \le 2(n - 1)/3 \le \delta(\overline{G[S_1]}) + 1,$$
by Lemma 5 we have that $K$ can be embedded to $\overline{G[S_1]}$ such that $N_K(x)$ is mapped to $k = d_K(x)$ distinct vertices $y_1, y_2, \dots, y_k$ in $\overline{N_{S_1}}(w)$. Note this is possible since $d_{K}(x) \le \overline{d_{S_1}}(w)$. This yields a successful embedding of $T_n$ in $\overline{G}$, which is a contradiction since $\overline{G}$ does not contain $T_n$ as a subgraph.

Now suppose that $d_{K}(x) > \overline{d_{S_1}}(w) \ge n - \delta(\overline{G[S_2]}) - 1$. Observe that $K$ is the union of $d_K(x)$ connected components of $T_n - x$. Let $C$ be the union of $d_K(x) - \overline{d_{S_1}}(w)$ of these connected components. Let $K' = K \backslash C$ and $H' = H \cup C$. Note that $d_{K'}(x) = \overline{d_{S_1}}(w) \ge n - \delta(\overline{G[S_2]}) - 1$ and that
$$n - \delta(\overline{G[S_2]}) - 1 \le d_{K'}(x) \le |K'| \le 2(n - 1)/3 < \delta(\overline{G[S_1]}) + 1.$$
The lower bound on $|K'|$ implies that
$$|H' \cup \{ x \}| = n - |K'| \le \delta(\overline{G[S_2]}) + 1.$$
Now applying the embedding described above to $K'$ and $H'$ in place of $K$ and $H$ yields the same contradiction.

The same method shows that each $w \in S_1$ satisfies that $\overline{d_{S_2}}(w) < n - \delta(\overline{G[S_1]}) - 1$. Specifically, if $\overline{d_{S_2}}(w) \ge n - \delta(\overline{G[S_1]}) - 1$ for some $w \in S_1$, embedding the tree $K \cup \{ x \}$ to $\overline{G[S_1]}$ with $x$ mapped to $w$ and embedding the forest $H$ to $\overline{G[S_2]}$ as in the argument above yields a contradiction. This proves the claim.
\end{proof}

From this point forward in the proof of Theorem 4, let $m = 2 \ell + 1$. The next two claims together complete the proof that the sets $S_1$ and $S_2$ induce a nearly complete bipartite subgraph of $G$, further showing that the structure of $G$ is close to that of the extremal graph $K_{n-1,n-1}$.

\begin{claim}
If the vertices $x,y \in S_1$, then $|N_{S_2}(x) \cap N_{S_2}(y)| > \ell$ and, if $x,y \in S_2$, then $|N_{S_1}(x) \cap N_{S_1}(y)| > \ell$.
\end{claim}

\begin{proof}
If $x, y \in S_1$, then since $\overline{d_{S_2}}(x), \overline{d_{S_2}}(y) < n - \delta(\overline{G[S_1]}) - 1$ we have that
\begin{align*}
|N_{S_2}(x) \cap N_{S_2}(y)| &\ge |S_2| - \overline{d_{S_2}}(x) - \overline{d_{S_2}}(y) \\
&> |S_2| - 2(n - \delta(\overline{G[S_1]}) - 1).
\end{align*}
Similarly, if $x, y \in S_2$, then since $\overline{d_{S_1}}(x), \overline{d_{S_1}}(y) < n - \delta(\overline{G[S_2]}) - 1$ we have that
\begin{align*}
|N_{S_1}(x) \cap N_{S_1}(y)| &\ge |S_1| - \overline{d_{S_1}}(x) - \overline{d_{S_1}}(y) \\
&> |S_1| - 2(n - \delta(\overline{G[S_2]}) - 1).
\end{align*}
We will show that both of these lower bounds are at least $\ell = (m-1)/2$ when $n \ge 25m$. Assume that $n \ge \frac{1}{2}c^2 m$ and note the following inequalities
\begin{align*}
|S_1| &\ge n - \sqrt{\frac{1}{2}(m-3)n} > \left( 1 - \frac{1}{c} \right) n, \\
\delta(\overline{G[S_1]}) &\ge n - \sqrt{2(m-3)n} > \left( 1 - \frac{2}{c} \right) n, \\
d &\ge n - \sqrt{\frac{1}{2}(m-3)n} + 1 > \left( 1 - \frac{1}{c} \right) n \ge \frac{1}{2}c(c-1)m, \\
|S_2| &\ge d - \sqrt{\frac{1}{2}(m-3)d} > \left( 1 - \frac{1}{\sqrt{c(c-1)}} \right) d \\
&> \left( 1 - \frac{1}{c} - \frac{\sqrt{c-1}}{c \sqrt{c}} \right) n, \quad \text{and} \\
\delta(\overline{G[S_2]}) &\ge d - \sqrt{2(m-3)d} \ge \left( 1 - \frac{1}{c} - \frac{2\sqrt{c-1}}{c \sqrt{c}} \right) n.
\end{align*}
These inequalities imply that
\begin{align*}
|S_2| - 2(n - \delta(\overline{G[S_1]}) - 1) &> \left( 1 - \frac{5}{c} - \frac{\sqrt{c-1}}{c \sqrt{c}} \right) n \ge \frac{n}{2c^2} \ge \frac{m}{2} > \ell \quad \text{and} \\
|S_1| - 2(n - \delta(\overline{G[S_2]}) - 1) &> \left( 1 - \frac{3}{c} - \frac{4\sqrt{c-1}}{c \sqrt{c}} \right) n \ge \frac{n}{2c^2} \ge \frac{m}{2} > \ell
\end{align*}
as long as we have the inequalities
\begin{align*}
1 - \frac{5}{c} - \frac{\sqrt{c-1}}{c \sqrt{c}} &\ge \frac{1}{2c^2} \quad \text{and} \\
1 - \frac{3}{c} - \frac{4\sqrt{c-1}}{c \sqrt{c}} &\ge \frac{1}{2c^2}.
\end{align*}
Since $\sqrt{c-1} < \sqrt{c}$, rearranging yields that these inequalities hold if $2c(c - 7) \ge 1$, which is true when $c^2 = 50$. The claim follows.
\end{proof}

\begin{claim}
There are no edges in $G[S_1]$ and $G[S_2]$.
\end{claim}

\begin{proof}
Now assume for contradiction that there is an edge $xy$ in $G[S_1]$. Let $z_1, z_2, \dots, z_{\ell + 1}$ be any distinct vertices of $S_1$ such that $z_1 = x$ and $z_{\ell + 1} = y$. By Claim 12, there are at least $\ell$ vertices of $S_2$ adjacent to both $z_i$ and $z_{i+1}$ for each $1 \le i \le \ell$. Therefore we may choose distinct vertices $w_1, w_2, \dots, w_\ell$ in $S_2$ such that $w_i$ is adjacent to both $z_i$ and $z_{i+1}$ for all $1 \le i \le \ell$. Now note that the vertices $z_1, w_1, z_2, w_2, \dots, z_\ell, w_\ell, z_{\ell + 1}$ form a cycle of length $2\ell + 1 = m$ in $G$, which is a contradiction. A symmetric argument shows that there are no edges in $G[S_2]$. This completes the proof of the claim.
\end{proof}

Now let $U = V(G) \backslash (S_1 \cup S_2)$. The next claim is the final ingredient necessary to construct a successful embedding of $T_n$ to $\overline{G}$.

\begin{claim}
Each $w \in U$ is adjacent to vertices in at most one of $S_1$ and $S_2$.
\end{claim}

\begin{proof}
Assume for contradiction that some vertex $w \in U$ is adjacent to $x \in S_1$ and $y \in S_2$. Since $d_{S_1}(y) = |N_{S_1}(y)| > \ell \ge 2$ by Claim 12, $y$ has a neighbor $z \in S_1$ where $z \neq x$. Now let $z_1, z_2, \dots, z_\ell$ be any distinct vertices of $S_1$ such that $z_1 = z$ and $z_\ell = x$. By Claim 12, it follows that there are at least $\ell$ vertices of $S_2$ adjacent to both $z_i$ and $z_{i+1}$ for each $1 \le i \le \ell - 1$. Therefore we may choose distinct vertices $w_1, w_2, \dots, w_{\ell - 1}$ in $S_2$ such that $w_i$ is adjacent to both $z_i$ and $z_{i+1}$ and $w_i \neq y$ for all $1 \le i \le \ell - 1$. Now note that the vertices $w, y, z_1, w_1, z_2, w_2, \dots, z_{\ell - 1}, w_{\ell - 1},  z_\ell$ forms a cycle of length $2 + 2(\ell - 1) + 1 = m$ in $G$, which is a contradiction. This proves the claim.
\end{proof}

By Claim 14, there are sets $U_1$ and $U_2$ such that $U = U_1 \cup U_2$ and no vertex in $U_i$ is adjacent to a vertex of $S_i$ for $i = 1, 2$. It follows that $|S_1 \cup U_1| + |S_2 \cup U_2| \ge |V(G)| = 2n - 1$ and therefore either $|S_1 \cup U_1| \ge n$ or $|S_2 \cup U_2| \ge n$.

First suppose that $|S_1 \cup U_1| \ge n$. Since $T_n$ is a tree, it is bipartite and admits a bipartition $V(T_n) = A \cup B$ where $|A|\ge |B|$ and thus $|A| \ge n/2$. Now consider the following embedding of $T_n$ to $\overline{G[S_1 \cup U_1]}$. If $|S_1| \ge n$ then map the vertices of $T_n$ arbitrarily to distinct vertices in $S_1$. Otherwise, map $n - |S_1|$ vertices in $A$ to distinct vertices in $U_1$ and the remaining $|S_1|$ vertices of $T_n$ to distinct vertices in $S_1$. Note that this is possible since $n - |S_1| \le \sqrt{\frac{1}{2}(m-3)n} \le n/2 \le |A|$ since $n \ge 2m - 6$. Since each vertex in $S_1$ is not adjacent to all other vertices in $S_1 \cup U_1$ and $A$ is an independent set of $T_n$, this is a valid embedding. This contradicts the fact that $\overline{G}$ does not contain $T_n$ as a subgraph. We arrive at a symmetric contradiction when $|S_2 \cup U_2| \ge n$. This proves Theorem 4.

\section{Conclusions and Future Work}

The primary direction for further work is to determine exactly the number $n_0(m)$. Our work and the path-odd cycle result of Faudree, Lawrence, Parsons and Schelp in \cite{faudree1974path} show that $2m/3 - 1 \le n_0(m) \le 25m$ \cite{faudree1974path}. Another possible direction for future work would be to extend the methods shown here to families of sparse graphs other than trees, such as unicyclic graphs.

\section*{Acknowledgements}

This research was conducted at the University of Minnesota Duluth REU and was supported by NSF grant 1358695 and NSA grant H98230-13-1-0273. The author thanks Joe Gallian for suggesting the problem and Levent Alpoge and Joe Gallian for helpful comments on the manuscript. The author also thanks P\'{e}ter Csikv\'{a}ri for suggesting Lemma 7 to improve the bound in Claim 8.

\end{document}